\documentclass[12pt, a4paper]{amsart}
\usepackage{graphicx}

\usepackage[a4paper, top=3cm, left=3cm, right=2cm, bottom=2.5cm]{geometry}
\usepackage{amsmath, amsthm}
\usepackage{amssymb}
\usepackage{multicol,comment}
\usepackage{amsfonts}
\usepackage{color, enumerate, fancyhdr, dsfont}
\usepackage[utf8]{inputenc}

\usepackage[normalem]{ulem}

    \newcommand{\Iwf}{\mathcal{I}}

    \newcommand{\Mwf}{\mathcal{M}}
    \newcommand{\Nwf}{\mathcal{N}}

     \newcommand{\pw}{\mathrm{pow}}
    \newcommand{\bfrak}{\mathfrak{b}}
    \newcommand{\cfrak}{\mathfrak{c}}
    \newcommand{\dfrak}{\mathfrak{d}}

    \newcommand{\menos}{\smallsetminus}

    \newcommand{\add}{\mbox{\rm add}}
    \newcommand{\cov}{\mbox{\rm cov}}
    \newcommand{\non}{\mbox{\rm non}}
    \newcommand{\cof}{\mbox{\rm cof}}

    \newcommand{\Por}{\mathds{P}}
    \newcommand{\Qor}{\mathds{Q}}

    \newcommand{\Fn}{\mathrm{Fn}}
    \newcommand{\supcof}{\mathrm{supcof}}
    \newcommand{\minadd}{\mathrm{minadd}}
    
    \newcommand{\hgt}{\mathrm{ht}}

\makeatother

\title[Yorioka's characterization of the cofinality of the strong measure zero]{Yorioka's characterization of the cofinality of the strong measure zero ideal and its independency from the continuum}

\author{Miguel A. Cardona}
\address{Institute of Discrete Mathematics and Geometry, TU Wien, Wiedner Hauptstrasse 8--10/104 A--1040 Wien, Austria.}
\email{miguel.montoya@tuwien.ac.at}
\urladdr{https://www.researchgate.net/profile/Miguel\_Cardona\_Montoya}

\thanks{The author was partially supported by the Austrian Science Fund (FWF) P30666}

\subjclass[2010]{03E17, 03E35, 03E40.}

\keywords{Strong measure zero sets, cardinal invariants, Yorioka ideals.}

\begin{document}

\makeatletter
\def\@roman#1{\romannumeral #1}
\makeatother

\theoremstyle{plain}
  \newtheorem{theorem}{Theorem}[section]
  \newtheorem{corollary}[theorem]{Corollary}
  \newtheorem{lemma}[theorem]{Lemma}
  \newtheorem{prop}[theorem]{Proposition}
  \newtheorem{claim}[theorem]{Claim}
  \newtheorem{Fact}[theorem]{Fact}
  \newtheorem{exer}[theorem]{Exercise}
\theoremstyle{definition}
  \newtheorem{definition}[theorem]{Definition}
  \newtheorem{example}[theorem]{Example}
  \newtheorem{remark}[theorem]{Remark}
  \newtheorem{context}[theorem]{Context}
  \newtheorem{question}[theorem]{Question}
  \newtheorem{problem}[theorem]{Problem}
  \newtheorem{notation}[theorem]{Notation}
  \newtheorem*{theorem-non}{Theorem A}  
  \newtheorem*{main-theorem}{Main Theorem}

 \newcommand{\SNcal}{\mathcal{SN}}

  \newcommand{\azul}[1]{{\color{blue}#1}}
\newcommand{\rojo}[1]{{\color{red}#1}}
\newcommand{\tachar}[1]{{\color{red}\sout{#1}}}
\definecolor{amber}{rgb}{1.0,0.49,0.0}

\definecolor{ogreen}{RGB}{107,142,35}

\newcommand{\verde}[1]{{\color{ogreen}#1}}
\newcommand{\amber}[1]{{\color{amber}#1}}

\maketitle

\newcommand{\la}{\langle}
\newcommand{\ra}{\rangle}
\newcommand{\id}{\mathrm{id}}
\newcommand{\sig}{\boldsymbol{\Sigma}}
\newcommand{\cosig}{\boldsymbol{\Pi}}

\begin{abstract}
   In this paper we present a simpler proof of the fact that no inequality between $\cof(\SNcal)$ and $\cfrak$ can be decided in ZFC by using well-known tecniques and results.

\end{abstract}

\section{Introduction}\label{SecIntro}
Borel \cite{Borel} introduced the class of Lebesgue measure zero subsets of the real line 
called \textit{strong measure zero} sets, which we denote by $\SNcal$. The cardinal invariants associated with strong measure zero have been investigated. To summarize some of the results:

\begin{theorem-non}\label{theoremA} The following holds in ZFC
\begin{enumerate}
    \item[(i)] (Carlson \cite{Carlson}) $\add(\Nwf)\leq\add(\SNcal)$,
    \item[(ii)] $\cov(\Nwf)\leq\cov(\SNcal)\leq\cfrak$,
    \item[(iii)] (Miller \cite{Miller}) $\cov(\Mwf)\leq \non(\SNcal)\leq\non(\Nwf)$ and  $\add(\Mwf)=\min\{\bfrak, \non(\SNcal)\}$,
    \item[(iv)] (Osuga \cite{Osuga}) $\cof(\SNcal)\leq2^\dfrak$.
\end{enumerate}
Moreover, each of the following staments is consistent with ZFC
\begin{enumerate}
    \item[(v)] (Goldstern, Judah and Shelah \cite{GJS}) $\cof(\Mwf)<\add(\SNcal)$,
    \item[(vi)] (Pawlikowski \cite{P90}) $\cov(\SNcal)<\add(\Mwf)$,
    \item[(vii)]  $\cfrak<\cof(\SNcal)$ (from CH),
    \item[(viii)] (Yorioka \cite{Yorioka}) $\cof(\SNcal)<\cfrak$,
    \item[(ix)] (Laver \cite{Laver}) $\cof(\SNcal)=\cfrak$ (a consequence of Borel's conjecture).
\end{enumerate}
\end{theorem-non}

To prove (vii) and (viii) Yorioka gave a characterization of $\SNcal$, and to do this he introduced the $\sigma$-ideals $\Iwf_f$ parametrized by increasing functions $f\in\omega^\omega$, which we call \emph{Yorioka ideals} (see Definition \ref{DefYorioId}). These ideals are subideals of the null ideal $\Nwf$ and they include $\SNcal$ and $\SNcal=\bigcap\{\Iwf_f:f\in\omega^\omega\textrm{\ increasing}\}$. Even more, he proved that $\cof(\SNcal)=\dfrak_\kappa$ (see Definition \ref{dnkappa}) whenever $\add(\Iwf_f)=\cof(\Iwf_f)=\kappa$ for all increasing $f$. Although Yorioka's original result futher $\dfrak=\cov(\Mwf)=\kappa$, this can be ommited because $\add(\Nwf)\leq\minadd\leq\add(\Mwf)$ and $\cof(\Mwf)\leq\supcof\leq\cof(\Nwf)$ (see \cite{Osuga,CM}). 

In this work, we provide a simpler proof of the consistency of (viii), which also applies for the consistency of (vii) and (ix).

\begin{main-theorem}[{Yorioka \cite{Yorioka}}]
Let $\kappa$, $\nu$ be infinite cardinals such that $\aleph_1\leq\kappa=\kappa^{<\kappa}<\nu=\nu^\kappa$ and assume that $\lambda$ is a cardinal such that $\kappa\leq\lambda=\lambda^{\aleph_0}$. Then there is some poset $\Qor$ such that $\Vdash_{\Qor}$ $\add(\Nwf)=\add(\SNcal)=\cov(\SNcal)=\non(\SNcal)=\cof(\Nwf)=\kappa$, $\cof(\SNcal)=\dfrak_\kappa=\nu$ and $\cfrak=\lambda$. 
\end{main-theorem}

We also show that this $\Qor$ forces $\add(\SNcal)=\cov(\SNcal)=\non(\SNcal)=\kappa$.

\section{Proof the main theorem}

We first start with basic definitions and facts. Let $\kappa$ be an infinite cardinal. Let $f, g\in\kappa^\kappa$. Set $f\le^* g$ if $\exists\alpha<\kappa\forall\beta>\alpha(f(\beta)\leq g(\beta))$. Denote $\pw_k:\omega\to\omega$ the function defined by $\pw_k(i):=i^k$, and define the relation $\ll$ on $\omega^\omega$ as follows:
$f\ll g \text{\ iff } \forall{k<\omega}(f\circ\pw_k\leq^* g)$.

\begin{definition}\label{DefYorioId} For $\sigma \in (2^{<\omega})^\omega$ define
\[[\sigma]_\infty:=\{x \in 2^\omega:\exists^{\infty}{n < \omega}^{}(\sigma(n) \subseteq x)\}=\bigcap_{n<\omega} \bigcup_{m \geqslant n}[\sigma(m)]\]

and $\hgt_{\sigma}\in\omega^{\omega}$ by $\hgt_{\sigma}(i):=|\sigma(i)|$ for each $i<\omega$. Let $f\in\omega^\omega$ be a increasing function, set
\[\Iwf_{f}:=\{X\subseteq 2^{\omega}:\exists{\sigma \in (2^{<\omega})^{\omega}}(X \subseteq [\sigma]_\infty\text{\ and }h_{\sigma}\gg f )\}.\]
\end{definition}
Any family of the form $\Iwf_f$ with $f$ increasing is called a \textit{Yorioka ideal}, since Yorioka \cite{Yorioka} has proved that $\Iwf_f$ is a $\sigma$-ideal in this case, and $\mathcal{SN}=\bigcap\{\Iwf_f:f\,\textrm{increasing}\}$.
Denote 
\[\minadd=\min\{\add(\mathcal{I}_{f}):f\,\,\textrm{increasing}\},\,\,\, \supcof=\sup\{\cof(\mathcal{I}_{f}):f\,\,\textrm{increasing}\}\]

\begin{definition}\label{dnkappa}
Let $\kappa$ be a regular cardinals. Define the cardinal numbers $\bfrak_\kappa$ and $\dfrak_\kappa$ as follows:
\[\bfrak_\kappa=\min\{|F|:F\subseteq \kappa^\kappa\;\&\;\forall g\in \kappa^\kappa\exists
f\in F(f\not\le^* g)\}\,\,\textit{the (un)bounding number of}\,\,\kappa^\kappa\]  and  
\[\dfrak_\kappa=\min\{|D|:D\subseteq \kappa^\kappa\;\&\;\forall g\in \kappa^\kappa\exists
f\in D(g\le^* f)\}\,\,\textit{the dominating number of}\,\, \kappa^\kappa\] 

In particular, when $\kappa=\omega$, $\bfrak_\kappa$ and $\dfrak_\kappa$ are $\bfrak$ and $\dfrak$ respectively, well known as the \textit{(un)bounding number} and \textit{the dominating number}.
\end{definition}

%\begin{Fact}\label{Fcohen} Let $f\in\omega^\omega\cap V$ such that $f\not\leq^*1$. Then, $\Vdash_{\Cor}2^\omega\cap V\in\Iwf_f$. 
%\end{Fact}
%\begin{proof}
% Take Cohen forcing as $\Cor:=\{s\in\Seq_{<\omega}(2^{<\omega}):\forall k<|s|(s(n)=g(n))\}$. Let $\sigma_{G}\in(2^{<\omega})^\omega$ be a Cohen real added by $\Cor$. It's sufficient to prove that $2^\omega\cap V\subseteq [\sigma_G]_\infty$. Assume that $f\in2^\omega\cap V$. For each $n<\omega$, define the sets $D_n:=\{s\in\Cor:\exists m\geq n(s(n)(m)\subseteq f(m))\}$ which are dense, so $G$ intersects all of these yielding $\forall n\in\omega\exists m\geq n(\sigma_{G}(n)(m)\subseteq f(m))$.
%\end{proof}

%\begin{lemma}\label{Cohenr}
%Let $\nu$ be an uncountable regular cardinal. Then $\Vdash_{\Por_\nu} 2^\omega\cap V\in\SNcal$
%\end{lemma}
%\begin{proof} Without loss generality assume that $g\in\omega^\omega\cap V_\nu$ such that $g\not\leq^*1$ , so there is some $\alpha<\nu$ such that $g\in\omega^\omega\cap V_\alpha$. so in $V_\alpha$, $\Vdash_{\Cor}2^\omega\cap V\in\Jwf_g $ by Fact \ref{Fcohen}.
%\end{proof}

Set $\Fn_{<\kappa}(I,J):=\{p\subseteq I\times J:|p|<\kappa\,\,\textrm{and}\,\,p\, \textrm{ function}\}$ for sets $I$, $J$ and an infinite cardinal $\kappa$.

\begin{lemma}\label{cohenk}
Let $\nu, \kappa$ be infinite cardinals such that $\kappa^{<\kappa}=\kappa$ and $\nu>\kappa$. Then $\Fn_{<\kappa}(\nu\times\kappa,\kappa)\Vdash\dfrak_\kappa\geq\nu$.
\end{lemma}
\begin{proof}
 Let $\vartheta<\nu$ and let $\{\dot{x}_\alpha:\alpha<\vartheta\}$ be a set of $\Fn_{<\kappa}(\nu\times\kappa,\kappa)$-names of functions in $\kappa^\kappa$. Since $\Fn_{<\kappa}(\nu\times\kappa,\kappa)$ is $(\kappa^{<\kappa})^+=\kappa^+$-cc we can find a subset $S$ of $\nu$ of size $<\nu$ such that $\dot{x}_\alpha$ is a $\Fn(S\times\kappa,\kappa)$-name for each $\alpha<\vartheta$.
 \begin{claim}\label{ccohenk}
 $\Fn_{<\kappa}(\kappa,\kappa)$ adds an unbounded function in $\kappa^\kappa$ over the ground model.
 \end{claim}
 \begin{proof}
Let $G$ be a $\Fn_{<\kappa}(\kappa,\kappa)$-generic set over $V$. Let $c:=c_G=\bigcup G\in\kappa^\kappa$ be the generic real added by $\Fn_{<\kappa}(\kappa,\kappa)$. Assume that $f\in\kappa^\kappa\cap V$. We will prove that $f\not\leq^* c$. To see this, for $\alpha<\kappa$, define the sets $D_\alpha:=\{p\in\Fn_{<\kappa}(\kappa,\kappa):\exists \beta> \alpha(p(\beta)> f(\beta))\}$ which is dense, so $G$ intersects all of these yielding $\forall \alpha<\kappa\exists \beta>\alpha(c(\beta)>f(\beta))$.
 \end{proof}
 By Claim \ref{ccohenk}, $\Fn_{<\kappa}(\nu\times\kappa,\kappa)$ forces that the $\kappa$-Cohen real at some $\xi\in\nu\menos S$ is not dominated by any $\dot{x}_\alpha$.
\end{proof}

As mentioned in the introduction,  $\add(\Nwf)\leq\minadd\leq\add(\Mwf)$ and $\cof(\Mwf)\leq\supcof\leq\cof(\Nwf)$, so we can reformulate Yorioka’s characterization of $\cof(\SNcal)$ as follows. 

\begin{theorem}[{Yorioka \cite{Yorioka}}]\label{Ylemma}
Let $\kappa$ be a regular uncountable cardinal. Assume that $\kappa=\minadd=\supcof$. Then $\cof(\SNcal)=\dfrak_\kappa$.
\end{theorem}

To prove our Main Thereom we need to preserve $\dfrak_\kappa$ for $\kappa$ regular. The following result show one condition under which it can be preserved. 

\begin{lemma}\label{fact}
Let $\kappa$ be a regular uncountable cardinal. Suppose that $\Por$ is a $\kappa$-cc poset. Then $\Vdash_{\Por}\dfrak_\kappa^V=\dfrak_\kappa$.
\end{lemma}
\begin{proof}
It is enough to show that $\Por$ is $\kappa^\kappa$-bounding\footnote{A poset $\Por$ is \emph{ $\kappa^\kappa$-bounding} if for any $p\in\Por$ and any $\Por$-name $\dot{x}$ of a member for  $\kappa^\kappa$, there are a function $z\in\kappa^\kappa$ and some $q\leq p$ that forces $\dot{x}(\alpha)\leq z(\alpha)$ for any $\alpha<\kappa$.} because  $\kappa^\kappa$-bounding posets preserve $\dfrak_\kappa$. Let $\dot{x}$ be a $\Por$-name for a member of $\kappa^\kappa$. We prove that $\forall\alpha<\kappa\exists z(\alpha)<\kappa(\Vdash_{\Por}\dot{x}(\alpha)<z(\alpha))$. Fix any $\alpha<\kappa$. Towards a contradiction, assume that $\forall\beta<\kappa\exists p_\beta\in\Por(p_\beta\Vdash_{\Por}\beta\leq\dot{x}(\alpha))$.
\begin{claim}\label{claim}
Assume that $\Por$ is $\kappa$-cc and $\{p_\alpha:\alpha<\kappa\}\subseteq\Por$. Then there is a $q\in\Por$ such that  $q\Vdash|\{\alpha<\kappa:p_\alpha\in\dot{G}\}|=\kappa$.
\end{claim}
\begin{proof}
To argue by contradiction assume that $\Vdash_{\Por}|\{\alpha<\kappa:p_\alpha\in\dot{G}\}|<\kappa$. Let $\dot{\beta}$ be a $\Por$-name such that $\Vdash\dot{\beta}\in\kappa$ and $\{\alpha<\kappa:p_\alpha\in\dot{G}\}\subseteq \dot{\beta}$. Fix a maximal antichain $A$ deciding $\dot{\beta}$ and a function $h:A\to\kappa$ such that $p\Vdash h(p)=\dot{\beta}$ for all $p\in A$. Set $\gamma:=\sup_{p\in A}h(p)<\kappa$. Since $\kappa$ is regular and $\Por$ is $\kappa$-cc, $\gamma<\kappa$, so $\Vdash_{\Por}\{\alpha<\kappa:p_\alpha\in\dot{G}\}\subseteq \gamma$. But $p_{\gamma+1}\Vdash\gamma+1\in\{\alpha<\kappa:p_\alpha\in\dot{G}\}\subseteq \gamma$, which is a contradiction.
\end{proof}

By Claim \ref{claim}, we can find a condition $q\in\Por$ such that $q\Vdash|\{\beta<\kappa:p_\beta\in\dot{G}\}|=\kappa$, so there are a $r\leq q$ and $\vartheta<\kappa$ such that $r\Vdash \dot{x}(\alpha)=\vartheta$, even more, we can find $s\leq r$ and $\varepsilon>\vartheta$ such that $s\Vdash p_\varepsilon\in\dot{G}$. Hence $s\Vdash \dot{x}(\alpha)=\vartheta<\varepsilon\leq\dot{x}(\alpha)$ because $p_{\varepsilon}\Vdash \varepsilon\leq \dot{x}(\alpha)$, which is a contradiction. 

Set $z\in\kappa^\kappa$ such that $\Vdash_{\Por}\dot{x}(\alpha)<z(\alpha)$ for any $\alpha<\kappa$. This $z$ works.
\end{proof}

Now we are ready to prove the Main Theorem.

\begin{proof}[Proof of the Main Theorem]
In $V$, we start with $\Por_0:=\Fn_{<\kappa}(\nu\times\kappa,\kappa)$. Note that $\Por_0$ is $\kappa^+$-cc and $<\kappa$-closed. Then $\Vdash_{\Por_0}\dfrak_\kappa=2^\kappa=\nu$ by Lemma \ref{cohenk}.

In $V^{\Por_0}$, let $\Por_1$ be the  FS iteration of amoeba forcing of length $\lambda\kappa$. Then, $\Vdash_{\Por_1}\add(\Nwf)=\cof(\Nwf)=\kappa$ and $\cfrak=\lambda$. In particular, $\Por_1$ $\add(\SNcal)=\non(\SNcal)=\kappa$ and $\minadd=\supcof=\kappa$. On the other hand, $\Por_1$ $\cov(\SNcal)=\kappa$ because the length of the FS iteration has cofinality $\kappa$ (see e.g. \cite[Lemma 8.2.6]{BJ}). Therefore, $\Vdash_{\Por_1}\add(\SNcal)=\cov(\SNcal)=\non(\SNcal)=\kappa$ and $\cof(\SNcal)=\dfrak_\kappa=\nu$ by Theorem \ref{Ylemma} and Lemma \ref{fact}.  
\end{proof}

\subsection{Acknowledgments} The author would like to thank my PhD advisor Diego A. Mej\'ia for the very useful discussions that helped this work to take its final form.

\section{Open problems} 
Quite recently, the author with Mej\'ia and Rivera-Madrid \cite{CMR} constructed a poset forcing 
$\non(\SNcal)<\cov(\SNcal)<\cof(\SNcal)$. This is first result where 3 cardianl invariants associated with $\SNcal$ are pairwise different, but its still unknown for 4, so we ask:
\begin{question}\label{Qadd}
   Is it consistent with ZFC that $\add(\SNcal)<\non(\SNcal)<\cov(\SNcal)<\cof(\SNcal)$?
\end{question}

In a work in progress, the author with Mej\'ia and Yorioka have improved methods and results known from \cite{Yorioka} to prove the consistency of $\cov(\SNcal)<\non(\SNcal)<\cof(\SNcal)$. However its still unknown the following problem.

\begin{question}\label{Qadd}
   Is it consistent with ZFC that $\add(\SNcal)<\cov(\SNcal)<\non(\SNcal)<\cof(\SNcal)$?
\end{question}
The method of $\kappa$-\textit{uf-extendable matrix iterations}, recently introduced by the author with Brendle and Mej\'ia \cite{BCM}, could be useful to answer the question above. For example they constructed a ccc poset forcing
\[\add(\Nwf)=\add(\Mwf)<\cov(\Nwf)=\non(\Mwf)<\cov(\Mwf)=\non(\Nwf)<\cof(\Mwf)=\cof(\Nwf).\]
In the same model, $\cov(\SNcal)=\cov(\Nwf)<\non(\SNcal)=\non(\Nwf)$ by Theorem A and because this model is obtained by a FS iteration of length with cofinality $\nu$ (where $\nu$ is the desired value for $\non(\Mwf)$), and it is well known that such cofinality becomes an upper bound of $\cov(\SNcal)$ (see e.g. \cite[Lemma 8.2.6]{BJ}). But it is unknown how to deal with $\add(\SNcal)$ and $\cof(\SNcal)$ in this context.

\bibliography{main}

\begin{thebibliography}{CMRM}

\bibitem[BCM]{BCM}
J\"{o}rg Brendle, Miguel~A. Cardona, and Diego~A. Mej\'ia.
\newblock Filter-linkedness and its effect on the preservation of cardinal
  characteristics.
\newblock \texttt{arXiv:1809.05004}.

\bibitem[BJ95]{BJ}
Tomek Bartoszy\'nski and Haim Judah.
\newblock {\em {Set Theory: On the Structure of the Real Line}}.
\newblock A K Peters, Wellesley, Massachusetts, 1995.

\bibitem[Bor19]{Borel}
\'Emile Borel.
\newblock Sur la classification des ensembles de mesure nulle.
\newblock {\em Bulletin de la Soci\'et\'e Math\'ematique de France},
  47:97--125, 1919.

\bibitem[Car93]{Carlson}
Timothy~J. Carlson.
\newblock Strong measure zero and strongly meager sets.
\newblock {\em Proc. Amer. Math. Soc.}, 118(2):577--586, 1993.

\bibitem[CM19]{CM}
Miguel~A. Cardona and Diego~A. Mej{\'{i}}a.
\newblock On cardinal characteristics of {Y}orioka ideals.
\newblock {\em MLQ}, 2019.
\newblock In press. \texttt{arXiv:1703.08634}.

\bibitem[CMRM]{CMR}
Miguel~A. Cardona, Diego~A. Mej\'ia, and Ismael~E. Rivera-Madrid.
\newblock The covering number of the strong measure zero ideal can be above
  almost everything else.
\newblock \texttt{arXiv:1902.01508v1}.

\bibitem[GJS93]{GJS}
Martin Goldstern, Haim Judah, and Saharon Shelah.
\newblock Strong measure zero sets without {C}ohen reals.
\newblock {\em J. Symbolic Logic}, 58(4):1323--1341, 1993.

\bibitem[Lav76]{Laver}
Richard Laver.
\newblock On the consistency of {B}orel's conjecture.
\newblock {\em Acta Math.}, 137(3-4):151--169, 1976.

\bibitem[Mil81]{Miller}
Arnold~W. Miller.
\newblock Some properties of measure and category.
\newblock {\em Trans. Amer. Math. Soc.}, 266(1):93--114, 1981.

\bibitem[Osu08]{Osuga}
Noboru Osuga.
\newblock The cardinal invariants of certain ideals related to the strong
  measure zero ideal.
\newblock {\em Ky\={o}to Daigaku S\=urikaiseki Kenky\=usho K\=oky\=uroku},
  1619:83--90, 2008.

\bibitem[Paw90]{P90}
Janusz Pawlikowski.
\newblock Finite support iteration and strong measure zero sets.
\newblock {\em J. Symbolic Logic}, 55(2):674--677, 1990.

\bibitem[Yor02]{Yorioka}
Teruyuki Yorioka.
\newblock The cofinality of the strong measure zero ideal.
\newblock {\em J. Symbolic Logic}, 67(4):1373--1384, 2002.

\end{thebibliography}
\bibliographystyle{bbl}

%\begin{thebibliography}{BML}
%     \bibitem[M13]{mejia} Mej\'ia, D.A.: \emph{Matrix iterations and Cichon's diagram.} Arch. Math. Logic (2012) doi:10.1007/s00153-012-0315-6
%\end{thebibliography}

\end{document}